\documentclass[10pt,letterpaper]{article}
\usepackage[square,sort,comma,numbers]{natbib}
\usepackage{bbm}
\usepackage{chngcntr}
\usepackage{multirow}

\usepackage[colorlinks=true,urlcolor=blue,citecolor=blue,linkcolor=blue,bookmarks=true,bookmarksopen=false]{hyperref}

\usepackage{algorithmic}
\usepackage[linesnumbered,ruled,vlined]{algorithm2e}

\usepackage[left=1in,top=1in,right=1in,bottom=1in]{geometry}

\newcommand{\exclude}[1]{}

\renewcommand{\Omega}{\varOmega}

\newcounter{commentcounter}
\long\def\symbolfootnote[#1]#2{\begingroup%
	\def\thefootnote{\fnsymbol{footnote}}\footnote[#1]{#2}\endgroup}

\usepackage{sibarticle}

\usepackage{amsmath}
\usepackage{kbordermatrix}

\usepackage{tikz}
\usetikzlibrary{arrows,shapes,snakes,automata,backgrounds,petri}
\usetikzlibrary{graphs}

\newcommand{\ignore}[1]{}

\usepackage{color}
\usepackage[symbol]{footmisc}
\usepackage{multirow}
\usepackage{booktabs}

\usepackage{rotating}

\usepackage{graphicx}
\usepackage{subfigure}
\usepackage{epstopdf}

\newcommand\tabref{Table~\ref}

\allowdisplaybreaks

\title{An Exact Method for Constrained Maximization of the Conditional Value-at-Risk of  a Class of Stochastic Submodular Functions}
\ShortTitle{CVaR Maximization of Stochastic Submodular Functions} \ShortAuthors{Wu and K\"u\c{c}\"ukyavuz}
\NumberOfAuthors{2} \FirstAuthor{ Hao-Hsiang Wu  }
\FirstAuthorAddress{  Department of Management Science, National Chiao Tung University, Hsinchu, Taiwan\\ 
	{\tt	hhwu2@nctu.edu.tw } }

	\SecondAuthor{Simge K\"u\c{c}\"ukyavuz\footnotemark[1]}
	\SecondAuthorAddress{
		Department of Industrial Engineering and Management Sciences,
		Northwestern University, Evanston, IL {\tt
			simge@northwestern.edu}
	}


\begin{document}

	\maketitle \centerline{\today}

	\begin{abstract} We consider a class of risk-averse submodular maximization problems (RASM) where the objective is the conditional value-at-risk (CVaR) of a random nondecreasing submodular function at a given risk level. We propose valid inequalities and an exact general method for solving RASM under the assumption that we have an efficient oracle that computes the CVaR of the random function. We demonstrate the proposed method on a stochastic set covering problem that admits an efficient CVaR oracle for the random coverage function.
	\end{abstract}
	
\section{Introduction}
\footnotetext[1]{Corresponding author} 
We consider  a class of risk-averse submodular maximization problems (RASM) recently  formulated by \citet{RiskSub2015}. Formally, let $V = \{1,\dots,n\}$ be a finite set, and $\bar \Omega$ be a probability space. We define an outcome mapping $\sigma: 2^{|V|} \times \bar \Omega \rightarrow \mathbb R$. The random outcome $\sigma(X): \bar \Omega \rightarrow \mathbb R$ is defined by $\sigma(X)(\omega) = \sigma(X,\omega)$ for all $X \subseteq {V}$ and $\omega \in \bar \Omega$. We assume that $\sigma(X)(\omega)$ is a nondecreasing submodular set function. With a slight abuse of notation, we refer to the submodular function $\sigma(X)$ for $X\subseteq V$  as $\sigma(x)$ for $x\in \{0,1\}^n$ interchangeably, where $x$ is the characteristic vector of $X$, and the usage will be clear from the context. We measure the risk associated with this function using  {\it conditional value-at-risk} (CVaR), where larger function values correspond to less risky random outcomes. In this context, risk measures are referred to as {\it acceptability functionals}. CVaR was first introduced by Artzner et al. \cite{cvar1999} and has been widely used, especially in finance, due to its desirable properties (e.g., coherence and tractability). 

Formally, let $[z]_+ = \max (z,0)$ be the positive part of a number $z \in \mathbb R$. Let $\eta$ be a variable that measures the {\it value-at-risk} (VaR) at a given risk level. For a given $\bar x$, the value-at-risk at a risk level $\alpha \in (0,1]$ is defined as
\begin{equation}\label{model:var_def}
\text{VaR}_{\alpha}(\sigma(\bar x)) =  \max \Big \{ \eta : \mathbb P( \sigma(\bar x) \ge \eta ) \ge 1-\alpha,  \eta \in \mathbb R \Big \}.  
\end{equation} 
For a given $\bar x $, the conditional value-at-risk at a risk level $\alpha \in (0,1]$ is defined  as  
\begin{equation}\label{model:cvar_def}
\text{CVaR}_{\alpha}(\sigma(\bar x))  =  \max \Big \{ \eta - \frac{1}{\alpha} \mathbb E([\eta-\sigma(\bar x)]_+) : \eta \in \mathbb R \Big \}, 
\end{equation} 
and it provides the conditional expected value of $\sigma(\bar x)$ that is no larger than the value-at-risk at the risk level $\alpha$ \cite{Rock2000}. In general, the risk level $\alpha$ is small, such as $\alpha = 0.05$ or $0.01$. Let $\mathcal{X}$ be the deterministic constraints on the variables $x$. Given a risk level $\alpha \in (0,1]$, a class of risk-averse submodular maximization problems (RASM) is defined as
\begin{equation}\label{model:max_cvar}
\max_{x \in \mathcal{X}}  \text{CVaR}_{\alpha}(\sigma(x)).
\end{equation} 

Let $\mathbb E[\sigma(x)]$ represent the expectation of $\sigma(x)$ with respect to $\omega$. In the next observation, we review properties of CVaR that will be useful  in developing valid inequalities for RASM.

\begin{observation}\label{prop:CVaR_sub}
	From the definition of CVaR in \eqref{model:cvar_def}, for a given $\bar x \in \mathcal{X}$, we have 
	\begin{itemize}				
		\item[(i)]
		$\text{CVaR}_{1}(\sigma(\bar x)) =\mathbb E[\sigma(\bar x)]$, and because $\mathbb E[\sigma(\bar x)]$  is submodular  (see, e.g., \cite{first2016}), so is $\text{CVaR}_{1}(\sigma(\bar x)) $.
		
		\item[(ii)] $\text{CVaR}_{\alpha_1}(\sigma(\bar x)) \le \text{CVaR}_{\alpha_2}(\sigma(\bar x))$ for $\alpha_1, \alpha_2 \in (0,1]$ and $\alpha_1 \le \alpha_2$.
	\end{itemize}	
\end{observation}    
Due to Observation \ref{prop:CVaR_sub}(i), when $\alpha=1$, problem \eqref{model:max_cvar} is equivalent to the  $\mathcal{NP}$-hard {\it risk-neutral} submodular maximization problem $\max_{x \in \mathcal{X}}  \mathbb E[\sigma(x)]$  (see, e.g., \cite{KKT03,SubMaxApp2012,VH93} for the associated applications).

Note that for a nondecreasing submodular function $\sigma(\cdot)$, Proposition 2 of \cite{NW81} shows that  the {\it submodular inequality}
\begin{equation}\label{eq:submod}
\sigma(x)\le \sigma(S)+\sum_{j\in V\setminus S} (\sigma(\{j\}\cup S)-\sigma(S))x_j, \quad \forall S\subseteq V,
\end{equation}
is valid.  \citet{Ahmed2011} and \citet{Yu2017} strengthen the submodular inequalities by lifting, under the assumption that the submodular utility function is strictly concave, increasing, and differentiable, assumptions we do not make in this paper. 
However, \citet{RiskSub2015} shows that CVaR of a stochastic submodular function is no longer submodular for any risk level $\alpha\in(0,1)$. The author proves that unless $\mathcal{P} = \mathcal{NP}$, there is no polynomial-time  approximation algorithm with a multiplicative error for RASM under some reasonable assumptions on the given risk level. Hence it is theoretically intractable to find a solution that is close to optimal in polynomial time without any assumption on the feasible region or the probability distribution.  Along this line of work,  \citet{Zhou2018} propose a sequential greedy algorithm for the CVaR maximization problem in \cite{RiskSub2015}. The authors show that the proposed algorithm gives a solution that is within a constant factor of optimal with an additional additive term that depends on the optimal value and  a parameter related  to the curvature of the submodular set function. However, the running time of this algorithm is exponential as it depends quadratically on the cardinality of the feasible set. 

Instead of solving the generic CVaR maximization problem in \cite{RiskSub2015} directly,  \citet{Ohsaka2017}  consider a specific risk-averse submodular optimization problem arising in social networks. The authors relax the problem by replacing the combinatorial decisions that determine a single set of a predetermined size with a choice of a portfolio  (convex combination) of sets of the given size. They also give a polynomial-time algorithm that obtains a portfolio with a CVaR value that has a provable guarantee with a specified probability. 
\citet{Wilder2017} generalizes these results and proposes an approximation algorithm  for maximizing the CVaR of a generic monotone {\it continuous} submodular function (or a portfolio of dicrete sets) with a worst-case guarantee within $(1-\frac{1}{e})$ factor of the optimal solution. That is, a function $F: \bar V \rightarrow \mathbb {R}$ is continuous submodular if and only if $F(a)+F(b) \ge F(a \vee b)+F(a \wedge b)$ for any $a,b \in \bar V \subseteq \mathbb R_+^n$, where $\bar V_i$ is a compact subset of $\mathbb R$, $\bar V = \prod_{i=1}^{n} \bar V_i \subseteq \mathbb R_+^n $, and notations $\vee$ and $\wedge$ denote the coordinate-wise minimum and maximum operations, respectively (see, e.g., \cite{Bian17} and \cite{ChenOnlineSub18}).  In contrast to this line of work that considers a portfolio of discrete sets as decision variables,  we consider the discrete case, where we must choose one set of decisions for implementability purposes. 
 In another line of work, \citet{Wu2017} propose an  exact method  for the minimal cost selection of $x \in \mathcal{X}$  that satisfies the chance constraint $\mathbb P( \sigma(x) \ge \tau )$ for a given $\tau$, where the authors assume that there exists a probability oracle for evaluating the chance constraint for a given solution. However, the exact method of \cite{Wu2017}  cannot be applied to the CVaR maximization problem (or its VaR maximization counterpart) directly. 

Note that for a given $\bar x \in \mathcal{X}$ and a finite probability space $(\Omega, 2^\Omega, \mathbb P)$ with a set of $N \in \mathbb N$ realizations (scenarios) $\Omega =\{\omega_1, \dots, \omega_N\}$, $\mathbb P (\omega_i) = p_i$ for $i = 1, \dots N$, CVaR is given by  \cite{Rock2000}
\begin{equation}\label{model:cvar_sub}
\text{CVaR}_{\alpha}(\sigma(\bar x))=\max \Big \{ \eta - \frac{1}{\alpha} \sum_{i \in [N]} p_i w_i : w_i \ge \eta - \sigma_i(\bar x), \quad \forall i \in [N], w \in \mathbb R^N_+, \eta \in \mathbb R \Big \},	
\end{equation} 
where  $[N]=\{1,\dots,N\}$ represents the set of the first $N$ positive integers, $\sigma_i(\bar x)=\sigma(x)(\omega_i)$, and $\sigma_q(\bar x)=\text{VaR}_{\alpha}(\sigma(\bar x))$ for at least one $\omega_q \in \Omega$. Therefore, when an oracle for evaluating CVaR is not available, we can take a sampling-based approach. 
To this end, in \cite{WuThesis}, we apply the stochastic submodular optimization methods of \cite{first2016} to problem \eqref{model:max_cvar} based on the CVaR definition \eqref{model:cvar_sub}. The resulting solutions are then tested  out-of-sample  and statistical performance guarantees are provided.

 In this paper, our  goal is to give (near-)optimal solutions for RASM without sampling. We start by considering high-quality (ideally optimal) feasible solutions to RASM under a true (non-trivial) distribution of the uncertain parameters assuming that there is an efficient oracle to evaluate the true value of CVaR of the random function at a given risk level. To solve RASM, we propose a decomposition method with various classes of valid inequalities. We demonstrate our proposed methods on a risk-averse set covering problem (RASC) that admits an efficient CVaR oracle in our computational study.

\section{RASM with an Exact CVaR Oracle} \label{sec:RASM-exact}

In this section, we assume that for a given incumbent solution $\bar x \in \mathcal X$, we have an efficient oracle that computes $\text{CVaR}_{\alpha}(\sigma(\bar x))$ exactly. Under this assumption, we solve problem \eqref{model:max_cvar} without sampling by using a two-stage optimization model and an exact algorithm with various valid inequalities. In the proposed method, we solve a relaxed master problem (RMP) at any iteration  in the form
\begin{equation}\label{model:max_cvar_master_2}
	\max \{ \psi:(x,\psi) \in \mathcal{C},x \in \mathcal{X}, \psi \in \mathbb R\}, 
\end{equation}
where  $\psi$ is a variable that is an upper bounding approximation of $\text{CVaR}_{\alpha}(\sigma(x))$ for a solution $x \in \mathcal{X}$, and $ \mathcal{C}$ is a set of optimality cuts to be defined later. It is an approximation in that not all necessary optimality cuts may have been generated at an intermediate iteration.  A \emph{full} master problem includes all  optimality cuts in  $ \mathcal{C}$ such that for any $x\in \mathcal X$, we have $\psi\le \text{CVaR}_{\alpha}(\sigma (x))$. Therefore, solving the full master problem is equivalent to solving the original problem \eqref{model:max_cvar}.
In Algorithm \ref{alg:DCG_2}, we propose a delayed constraint generation algorithm for solving problem \eqref{model:max_cvar} using RMP \eqref{model:max_cvar_master_2}. The algorithm starts with a set of optimality cuts $\mathcal{C}$ (could be empty). In the while loop, we solve RMP \eqref{model:max_cvar_master_2} and obtain an incumbent solution (Line \ref{dcg2_master}). Based on the incumbent solution, we add an optimality cut to RMP \eqref{model:max_cvar_master_2} (Line \ref{dcg2_add_cut}). In this algorithm, $\epsilon$ is
a user-defined optimality tolerance. Let UB be the upper bound obtained from the optimal objective value of RMP \eqref{model:max_cvar_master_2} at each iteration. Let LB  be the lower bound equal to $\text{CVaR}_\alpha(\sigma(\bar x))$, obtained from calling the CVaR oracle with the given incumbent solution $\bar x \in \mathcal X$ as input. If the optimality gap is below $\epsilon$, then we terminate the algorithm and return the near-optimal solution.   

\begin{algorithm}[htb]\label{alg:DCG_2}
	\SetAlgoLined
	Start with  an initial set of optimality cuts in $\mathcal{C}$ (could be empty), UB$ = \infty$ and LB$ = -\infty$\;
	
	\While{UB-LB $>
	 \epsilon$}
	{\do 	
		Solve RMP \eqref{model:max_cvar_master_2} and obtain $(\bar \psi, \bar x)$. \; \label{dcg2_master}		
		UB $ \leftarrow$ the optimal objective value of RMP \eqref{model:max_cvar_master_2}, LB$ \leftarrow \text{CVaR}_\alpha(\sigma(\bar x))$\; 
		Add an optimality cut to $\mathcal {C}$ \; \label{dcg2_add_cut}		 
	}				
	Output $\bar x$ as the optimal solution.
	\caption{A Delayed Constraint Generation Algorithm with a CVaR oracle}
\end{algorithm}

As we mentioned earlier,  $\text{CVaR}_{\alpha}(\sigma(x))$ is not submodular in $x$ even if $\sigma(x)$ is submodular \cite{RiskSub2015}.  Hence, there is no direct way use the submodular inequality \eqref{eq:submod} as a class of optimality cuts in $ \mathcal{C}$. Therefore,  in this section, we propose new optimality cuts that are valid for \eqref{model:max_cvar_master_2}. Throughout, we let $\mathbf {e}_{j}$ be a unit vector of dimension $|V|$ whose $j$th component is 1, and let $\mathbf{1}$ be a $|V|$-dimensional vector with all entries equal to 1. For a given $\bar x\in \mathcal X$,  we first propose an optimality cut given by
\begin{equation}\label{eq:new_cut}
\psi \le \text{CVaR}_{\alpha}(\sigma(\bar x)) +\sum_{j\in V\setminus \bar X} \big (\text{CVaR}_{1}(\sigma(\bar x+\mathbf {e}_{j}))-\text{CVaR}_{\alpha}(\sigma(\bar x))\big) x_j.
\end{equation} 
Before  formally proving the validity of inequality \eqref{eq:new_cut}, a few remarks are in order. It may be tempting to think that inequality \eqref{eq:new_cut} is in the form of a submodular inequality \eqref{eq:submod}. However, we highlight that the coefficients $ \big (\text{CVaR}_{1}(\sigma(\bar x+\mathbf {e}_{j}))-\text{CVaR}_{\alpha}(\sigma(\bar x))\big) $ for $j\in V\setminus \bar X$ are different from their submodular counterparts  $ \big (\text{CVaR}_{\alpha}(\sigma(\bar x+\mathbf {e}_{j}))-\text{CVaR}_{\alpha}(\sigma(\bar x))\big) $ for $j\in V\setminus \bar X$  and
because $\text{CVaR}_{\alpha}(\sigma(\cdot))$ is not submodular,  the latter coefficients are not valid. To provide some intuition for  the validity of the coefficients of the variables $x_j, j\in V\setminus \bar X$ in inequality \eqref{eq:new_cut}, we consider 
the upper bound on $\psi$ provided by the right-hand side of the inequality (recall that $\psi\le \text{CVaR}_{\alpha}(\sigma(x))$  for all $x\in \mathcal X$). First, observe that for $x=\bar x$, inequality \eqref{eq:new_cut} holds trivially, because $x_j=0$ for $j\in V\setminus \bar X$. Now consider the point $x=\bar x+\mathbf {e}_{j}$ for some $j\in V\setminus \bar X$. Then the right-hand side of inequality \eqref{eq:new_cut} is  $\text{CVaR}_{1}(\sigma(\bar x+\mathbf {e}_{j}))$, which is a valid upper bound on $\psi$ (because $\psi\le \text{CVaR}_{\alpha}(\sigma(\bar x+\mathbf {e}_{j}))\le \text{CVaR}_{1}(\sigma(\bar x+\mathbf {e}_{j}))$ from Observation \ref{prop:CVaR_sub} (ii)). The validity for other $x$ is proven by exploiting the properties  of $\text{CVaR}_{\alpha}(\sigma(x)) $ given in Observation \ref{prop:CVaR_sub} as we show next.

\begin{proposition}\label{prop:valid_new_cut1}
	Inequality \eqref{eq:new_cut} for a given $\bar x \in \mathbb {B}^n$   is valid for RMP \eqref{model:max_cvar_master_2}.
\end{proposition}    
\begin{proof}	
	Consider a  feasible point $( \hat \psi,  \hat  x)$ to the full master problem, in other words, let $\hat  x\in\mathcal X$ such that $\hat \psi \le  \text{CVaR}_{\alpha}(\sigma (\hat x))$. We show that $( \hat \psi,  \hat  x)$ satisfies  inequality \eqref{eq:new_cut} written for any $\bar x \in \mathcal X$. Let  $ \hat X = \{i \in V:{ \hat x_i} = 1\}$. From the definition of CVaR in \eqref{model:cvar_def}, it follows that because $\sigma(x)(\omega)$ is nondecreasing in $x$ for all $\omega \in \bar \Omega$,  $\text{CVaR}_{\alpha}(\sigma (x))$ is also nondecreasing in $x$. For the case that $\hat X \subseteq \bar X$, since $\text{CVaR}_{\alpha}(\sigma (x))$ is a monotonically nondecreasing function in $x$ and $\hat x_j = 0$ for all $j \in V\setminus \bar X$, we have $\hat  \psi  \le \text{CVaR}_{\alpha}(\sigma (\hat x))  \le \text{CVaR}_{\alpha}(\sigma (\bar x))$, which shows that inequality \eqref{eq:new_cut} is valid for $\hat X \subseteq \bar X$.  
	For the case that $\hat X \setminus \bar X \neq \emptyset$, we select an arbitrary  $j' \in \hat X \setminus \bar X$, where $\hat x_{j'} = 1$. Then 
		\begin{subequations}
			\begin{align}
			\hat  \psi & \le \text{CVaR}_{\alpha}(\sigma (\hat x)) \nonumber \\
			& \le \text{CVaR}_{1}(\sigma (\hat x)) \label{ceq:2}\\
			& \le  \text{CVaR}_{1}(\sigma (\bar x)) + \sum_{j \in V \setminus \bar X } [\text{CVaR}_{1}(\sigma (\bar x +\mathbf {e}_{j}))-\text{CVaR}_{1}(\sigma (\bar x))] \hat x_j\label{ceq:3}\\			
			& =  \sum_{j \in V \setminus \bar X } \text{CVaR}_{1}(\sigma (\bar x +\mathbf {e}_{j})) \hat x_j-  \sum_{j \in V \setminus (\bar X \cup \{j'\})} \text{CVaR}_{1}(\sigma (\bar x))\hat x_j \label{ceq:4}\\
			& \le  \sum_{j \in V \setminus \bar X } \text{CVaR}_{1}(\sigma (\bar x +\mathbf {e}_{j})) \hat x_j-  \sum_{j \in V \setminus (\bar X \cup \{j'\})} \text{CVaR}_{\alpha}(\sigma (\bar x))\hat x_j \label{ceq:5}\\
			& = \bigg ( \sum_{j \in V \setminus \bar X } \text{CVaR}_{1}(\sigma (\bar x +\mathbf {e}_{j})) \hat x_j-  \sum_{j \in V \setminus \bar X} \text{CVaR}_{\alpha}(\sigma (\bar x))\hat x_j \bigg ) + \text{CVaR}_{\alpha}(\sigma (\bar x))\hat x_{j'}\label{ceq:6}\\	
			& =  \text{CVaR}_{\alpha}(\sigma(\bar x)) +\sum_{j\in V\setminus \bar X} (\text{CVaR}_{1}(\sigma(\bar x+\mathbf {e}_{j}))-\text{CVaR}_{\alpha}(\sigma(\bar x))) \hat x_j. \label{ceq:7}
			\end{align}
		\end{subequations} 
		Inequality \eqref{ceq:2} follows from Observation \ref{prop:CVaR_sub} (ii) that $\text{CVaR}_{\alpha}(\sigma(x))$   is  nondecreasing  in $\alpha$.  
		Inequality \eqref{ceq:3} follows from the submodular inequality \eqref{eq:submod} written for the monotone submodular function  $\text{CVaR}_{1}(\sigma(x)) $   (see Observation \ref{prop:CVaR_sub} (i)) and for $S=\bar X$ evaluated at the point $x=\hat x$. Arranging terms in inequality \eqref{ceq:3} and recalling that $\hat x_{j'} = 1$, we obtain equality \eqref{ceq:4}.
		Inequality \eqref{ceq:5} follows  from Observation \ref{prop:CVaR_sub} (ii), this time observing that $-\text{CVaR}_{\alpha}(\sigma(x))$   is  nonincreasing  in $\alpha$. To obtain equality \eqref{ceq:6}, we add $\text{CVaR}_{\alpha}(\sigma (\bar x))\hat x_{j'}-\text{CVaR}_{\alpha}(\sigma (\bar x))\hat x_{j'}$ to inequality \eqref{ceq:5} and reorganize the terms. Finally, equality \eqref{ceq:7} follows from $\hat x_{j'} = 1$. This completes the proof. 
\end{proof}

Next, we introduce a class of valid inequalities obtained by a sequential lifting procedure \cite[][Proposition 1.1 in Section II.2.1]{NW88}. Given  $\bar x\in \mathcal X$, which is a characteristic vector of the set $\bar X$, consider a restriction with $x_j=0$ for $j\in V\setminus \bar X$. For this restriction, we know that the \emph{base inequality}  $\psi\le  \text{CVaR}_{\alpha}(\sigma(\bar x))$ is valid, because for any $x$ satisfying this restriction, we have $X\subseteq \bar X$, and  $\psi= \text{CVaR}_{\alpha}(\sigma( x)) \le  \text{CVaR}_{\alpha}(\sigma(\bar x))$ due the property that $\text{CVaR}_{\alpha}(\sigma(x))$ is monotonically nondecreasing in $x$.  However, this base inequality is not valid when the restriction is lifted. To obtain a valid inequality, we sequentially lift the base inequality with the variables $x_j, j\in V\setminus \hat X$.
Let $j_1,\dots,j_{r}$ be an ordering of the elements in $V \setminus \bar X$, where $r = |V \setminus \bar X|$. Sequential lifting following this order produces a valid inequality
\begin{equation}\label{eq:new_cut_uplift}
\psi \le  \text{CVaR}_{\alpha}(\sigma(\bar x))  +\sum_{i = 1}^{r}  \delta_{j_i}(\bar x) x_{j_i},
\end{equation} 
where $\delta_{j_t}(\bar x)$ for $t=1,\dots,r$ is an \emph{exact} lifting function given by the  $t$-th lifting problem 
\begin{subequations}
	\label{eq:uplift}
	\begin{align}
	\delta_{j_t}(\bar x) =   -\text{CVaR}_{\alpha}(\sigma(\bar x)) +
 \max~~&	\text{CVaR}_{\alpha}(\sigma(x)) - \sum_{i = 1}^{t-1} \delta_{j_i}(\bar x) x_{j_i}  \label{eq:uplift-obj} \\
	\text{s.t.}~~ 		
	& x_{j_t} = 1 \label{eq:uplift_1}\\
	& x_{j_i} = 0 , \quad i=t+1,\dots, r \label{eq:uplift_3}\\				 
	& x \in \mathcal{X}. \label{eq:uplift_5}
	\end{align}
\end{subequations}
In \eqref{eq:uplift-obj}, we use the convention that for $t=1$, the term $\sum_{i=1}^{t-1} (\cdot)=0$. Note that in the $t$-th lifting problem, to obtain the coefficient of the variable $x_{j_t}$ in inequality \eqref{eq:new_cut_uplift}, we let $x_{j_t}=1$ in constraint \eqref{eq:uplift_1}, we remove the restriction on the variables preceding this variable in the lifting sequence, while keeping the restriction that the variables following $x_{j_t}$ in the lifting sequence are fixed to zero in constraint \eqref{eq:uplift_3}.  
The exact lifting problem is hard to solve since it is related to the submodular maximization problem \eqref{model:max_cvar}. Instead of solving problem \eqref{eq:uplift} exactly, we propose to solve a relaxation that provides an upper bound for the lifting coefficients. Then, we can obtain another valid inequality by using the upper bounds on the coefficients instead of the exact coefficients. We describe this approach next.  

Given  $\bar x\in \mathcal X$, we define $\bar \delta_{j_t}(\bar x)$ as an  upper bound of $ \delta_{j_t}(\bar x)$ for $t=1,\dots,r$, where $\bar \delta_{j_t}(\bar x)$ is given by solving  the following relaxation of the exact lifting problem \eqref{eq:uplift}:
\begin{equation}\label{eq:relax_uplift}
	\bar \delta_{j_t}(\bar x) = \max\{\text{CVaR}_{\alpha}(\sigma(x)),\eqref{eq:uplift_1}, \eqref{eq:uplift_3},x \in \mathbb B^n\} -  \text{CVaR}_{\alpha}(\sigma(\bar x));
\end{equation}
here we relax constraints \eqref{eq:uplift_5} and remove the term $-\sum_{i = 1}^{t-1} \delta_{j_i}(\bar x) x_{j_i}$ from the objective function in problem \eqref{eq:uplift}. This is a valid relaxation, because $\delta_{j_i}(x)\ge 0$  when $\text{CVaR}_{\alpha}(\sigma(x))$ is monotonically nondecreasing in $x$. Furthermore,  because $x_{j_t} = 1$ and $x_{j_i} = 0$ for $i=t+1,\dots, r$, the feasible solution with $x^*_{j_i} = 1$ for $i=1,\dots, t$ has the largest $\text{CVaR}_{\alpha}(\sigma(x^*))$ in the objective function of \eqref{eq:relax_uplift}. Thus, the optimal solution of the relaxed lifting problem \eqref{eq:relax_uplift} is given by $\bar x^{j_t} = \mathbf{1} -\sum_{i=t+1}^{r}\mathbf{e}_{j_i}$, where we can use an efficient  oracle for $\text{CVaR}_{\alpha}(\sigma(x))$ to obtain the optimal value of $\bar \delta_{j_t}(\bar x)$ efficiently. 
\begin{proposition}\label{prop:valid_relax_upliftcut}
	Given $\bar x \in 
	\mathcal X$, its support $\bar X$ and  an ordering of $V \setminus \bar X$ given by $\{j_1,\ldots,j_r\}$, inequality 
	\begin{equation}\label{eq:relax_cut_uplift}
	\psi \le  \text{CVaR}_{\alpha}(\sigma(\bar x))  +\sum_{i = 1}^{r}  \bar \delta_{j_i}(\bar x) x_{j_i},
	\end{equation} 
	where $\bar \delta_{j_i}(\bar x) = \text{CVaR}_{\alpha}(\sigma(\bar x^{j_i}))- \text{CVaR}_{\alpha}(\sigma(\bar x)) $,
	is valid for RMP \eqref{model:max_cvar_master_2}.
\end{proposition}    
\begin{proof}
	Consider a feasible point $( \hat \psi,  \hat  x)$. We show that $( \hat \psi,  \hat  x)$ satisfies  inequality \eqref{eq:relax_cut_uplift}   written for  $\bar x \in \mathbb {B}^n$.
		\begin{subequations}
	\begin{align}
		\psi& \le  \text{CVaR}_{\alpha}(\sigma(\bar x))  +\sum_{i = 1}^{r}  \delta_{j_i}(\bar x) \hat x_{j_i}\label{eq:relax_cut_uplift1}\\
		& \le  \text{CVaR}_{\alpha}(\sigma(\bar x))  +\sum_{i = 1}^{r}  \bar \delta_{j_i}(\bar x) \hat x_{j_i}.\label{eq:relax_cut_uplift2}
	\end{align}
\end{subequations} 
	Inequality \eqref{eq:relax_cut_uplift1} follows from the valid inequality \eqref{eq:new_cut_uplift}. Inequality \eqref{eq:relax_cut_uplift2} follows from $\bar \delta_{j_i}(\bar x) \ge \delta_{j_i}(\bar x)$ for $i=1,\dots,r$ since problem \eqref{eq:relax_uplift} is a relaxation of problem \eqref{eq:uplift}. This completes the proof.		
\end{proof} 
While inequality \eqref{eq:relax_cut_uplift} is derived by approximate lifting, next  we provide some intuition on  the validity of the coefficients of the variables $x_j, j\in V\setminus \bar X$ in inequality \eqref{eq:relax_cut_uplift}. Consider 
the upper bound on $\psi\le\text{CVaR}_{\alpha}(\sigma ( x))$ for some $x\in \mathcal X$ provided by the right-hand side of the inequality. First, observe that for $x=\bar x$, inequality \eqref{eq:relax_cut_uplift} holds trivially, because $x_j=0$ for $j\in V\setminus \bar X$. Now consider the point $x=\bar x+\mathbf {e}_{j_t}$, for some $t\in\{1,\dots,r\}$. Then the right-hand side of inequality \eqref{eq:relax_cut_uplift} is  $\text{CVaR}_{\alpha}(\sigma(\bar x^{j_t}))$, which is a valid upper bound on $\psi=\text{CVaR}_{\alpha}(\sigma(\bar x+\mathbf {e}_{j_t}))$ from Observation \ref{prop:CVaR_sub} (ii). The approximate lifting argument establishes that the inequality is valid for other choices of $x$ as well, because the coefficients of the variables are no smaller than those obtained from an exact lifting problem. 

In Algorithm \ref{alg:greedy_uplift},   we propose a greedy method that generates an inequality \eqref{eq:relax_cut_uplift} given an incumbent $\bar x$. In the for loop, Line \ref{alg:gu_second_cond} determines the entry $j_i$ by choosing the candidate index $s$ for which $\bar \delta_{j_i = s}(\bar x)$ attains its smallest value, where the candidate $s$ has not been chosen previously. 
\begin{algorithm}[htb]\label{alg:greedy_uplift}
	\SetAlgoLined
 \label{alg:gu1} 
	$S \leftarrow V \setminus \bar X$ \;
	\For{$i =  1$ to $r$}
	{\do 
		\;			 			
		$j_i \leftarrow \arg \min_{s \in S} \bar \delta_{j_i = s}(\bar x)$ \; \label{alg:gu_second_cond}
		$x^* \leftarrow \mathbf{1} -\sum_{i=t+1}^{r}\mathbf{e}_{j_i}$\;
		$\bar \delta_{j_i}(\bar x) \leftarrow \bar \delta_{j_i}(x^*)$ \;
		$S \leftarrow S \setminus \{j_i\}$ \;		
	}				
	Output an inequality \eqref{eq:relax_cut_uplift} as the solution.
	\caption{GreedyUp-lifting($\alpha,\bar x$)}
\end{algorithm}

Finally, we show the correctness of Algorithm \ref{alg:DCG_2} based on the inequalities \eqref{eq:new_cut}  or  \eqref{eq:relax_cut_uplift}.

\begin{proposition}\label{prop:valid_algo1}
Algorithm \ref{alg:DCG_2} with optimality cuts  \eqref{eq:new_cut}  and/or  \eqref{eq:relax_cut_uplift} finitely converges to an optimal solution.
\end{proposition}
\begin{proof}
From Proposition \ref{prop:valid_new_cut1} and \ref{prop:valid_relax_upliftcut}, for all $\bar X \subseteq V$ and $x \in \mathcal X$, we have
\begin{equation*}
\psi \le \text{CVaR}_{\alpha}(\sigma(\bar x)) \le \text{CVaR}_{\alpha}(\sigma( x)) +\sum_{j\in V\setminus \bar X} \big (\text{CVaR}_{1}(\sigma(\bar x+\mathbf {e}_{j}))-\text{CVaR}_{\alpha}(\sigma(\bar x))\big) x_j,
\end{equation*} 
and
\begin{equation*}
\psi \le \text{CVaR}_{\alpha}(\sigma(\bar x)) \le  \text{CVaR}_{\alpha}(\sigma(x))  +\sum_{i = 1}^{r}  \bar \delta_{j_i}(\bar x) x_{j_i}.
\end{equation*} 
Hence, the following two linear integer programs,
\begin{equation*}
\max \{ \psi: \psi \le \text{CVaR}_{\alpha}(\sigma(x)) +\sum_{j\in V\setminus \bar X} \big (\text{CVaR}_{1}(\sigma(\bar x+\mathbf {e}_{j}))-\text{CVaR}_{\alpha}(\sigma(\bar x))\big) x_j \quad \forall \bar X \subseteq V, x \in \mathcal{X}, \psi \in \mathbb R\}
\end{equation*}
and
\begin{equation*}
\max \{ \psi: \psi \le  \text{CVaR}_{\alpha}(\sigma(x))  +\sum_{i = 1}^{r}  \bar \delta_{j_i}(\bar x) x_{j_i} \quad \forall \bar X \subseteq V, x \in \mathcal{X}, \psi \in \mathbb R\},
\end{equation*} 
 are equivalent to Problem \eqref{model:max_cvar}. Since the number of feasible solutions and the number of  inequalities  \eqref{eq:new_cut}  and  \eqref{eq:relax_cut_uplift} are finite, the result follows.
\end{proof}
Next we report our
computational experience with the proposed method.

\section{An Application: the Risk-Averse Set Covering Problem (RASC)} \label{RASC_sec}

We demonstrate the proposed methods on a variant of a risk-averse set covering problem (RASC). We represent RASC on a bipartite graph $G=(V_1 \cup V_2, E)$. There are two groups of nodes $V_1$ and $V_2$ in $G$, where all arcs in $E$ are from $V_1$ to $V_2$. Let $V_2:=\{1,\dots,m\}$ be a set of items. Let $S_j\subseteq V_2, j\in V_1:=\{1,\dots,n\}$ be a collection of $n$ subsets, where $\bigcup_{j=1}^n S_j=V_2$. 
There exists an arc $(i,j) \in E$  if  $j \in S_i$ for $i \in V_1$ representing the covering relationship. In  RASC, there is uncertainty on whether an arc appears in the graph. To formulate RASC, first consider a deterministic set covering problem with a feasibility set $\{x\in \mathbb{B}^{n} | \sum_{j\in V_1}  t_{ij} x_j \ge h_i, \quad \forall i \in V_2  \}$, where $h_i=1$ for all $i\in V_2$, and $ t_{ij}=1$ if $i \in S_j$; otherwise, $ t_{ij} = 0$ for all $i\in V_2\setminus S_j, j\in V_1$. We say that an item $i\in V_2$ is covered, if there exists $x_j=1$ for $j\in V_1$ and $t_{ij}=1$.  Now suppose that there is uncertainty on whether a chosen subset can cover an item, where constraint $t_{ij} x_j \ge h_i$ has random constraint coefficients $t_{ij}$ or  random right-hand side $h_i$ for $ i \in V_2, j \in V_1$. A related class of risk-averse problems, referred to as the probabilistic set covering problems,  consider a feasible set $ \{x\in \mathbb{B}^{n} | \mathbb P( \sum_{j\in V_1}  t_{ij} x_j \ge h_i, \quad \forall i \in V_2)\ge 1-\alpha \}$, where $\alpha$ is a given risk level and $t_{ij}$ and/or $h_i$ are random variables for $i\in V_2, j\in V_1$. Here the objective is to select a minimum cost selection of subsets such that the probability of covering all items is at least the risk threshold $1-\alpha$ (see, e.g., \cite{Beraldi2002b,Saxena2010,Fischetti2012,Ahmed2013}).  In contrast, in  this paper, we  assume that $h_i=1$ for all $i\in V_2$, and that there is uncertainty on $t_{ij}$ for all $i \in V_2$ and $j \in V_1$ and consider a different type of risk aversion, where we aim to choose at most $k$ subsets from the collection so that the CVaR of the number of covered items at a risk level $\alpha$ is maximized. 
 In this section, we consider an {\it independent probability coverage model}, where each node $j$ has an independent probability $a_{ij}$ of being covered by  node $i \in V_1$ for  $j\in S_i$, i.e., $\mathbb P(t_{ij}=1)=a_{ij}$.

Let $\sigma(x)$ be a random variable representing the number of covered items in $V_2$ for a given $x$, i.e., $\sigma(x):=|\{i\in V_2: \exists j\in V_1 \text{ with } x_j=1 \text{ and } t_{ij}=1\}|$. It is known that $\sigma(x)(\omega)$ is submodular for $\omega \in \Omega$ \cite[]{Wu2017}. Given an integer $k$ and $\alpha \in (0,1]$, problem \eqref{model:max_cvar} for RASC is 
\begin{equation}\label{model:maxcov_cvar}
\max\{\text{CVaR}_{\alpha}(\sigma(x)):\sum_{i \in V_1} x_i \le k,x\in\mathbb{B}^{n}\},	
\end{equation} 
which is in the form of \eqref{model:max_cvar}, where $\mathcal{X}$ is given by $\mathcal X= \{x\in\mathbb{B}^{n}: \sum_{i \in V_1} x_i \le k\}$. Next we propose an efficient CVaR oracle to evaluate the objective function in  problem \eqref{model:maxcov_cvar} for a given $x \in \mathcal{X}$ under the probability distribution of interest.
\begin{proposition} \label{prop:oracle}
 There exists a polyniamial-time oracle that computes the function $\text{CVaR}_{\alpha}(\sigma(x))$ for $x \in \mathcal{X}$ for RASC under the independent probability coverage model. 
\end{proposition}
\begin{proof}
We follow the notation described in \cite{Wu2017} to describe the CVaR oracle. From  \cite{Hoeffding1956,Samuels1965,Wang1993}, we know that function $\mathbb P( \sigma(x) = b)$ is equal to the probability mass function of a Poisson binomial distribution and use a dynamic program (DP)  to evaluate $A(x,i,j)$, which is defined as the probability that the selection $x$ covers $j$ nodes among the first $i$ nodes of $V_2$ for $0\le j\le i, i\in V_2$. \citet{Barlow1984,ZCAWZ14}, and  \citet{Wu2017}  use the DP  to  calculate $\mathbb P( \sigma(x) \ge b):=\sum_{j=b}^{m}A(x,m,j)$. Next, we show that we can use the same recursion  to evaluate $\text{CVaR}_{\alpha}(\sigma(x))$. 
From the definition of  $\text{VaR}_{\alpha}(\sigma(x))$ in \eqref{model:var_def}, we have $\text{VaR}_{\alpha}(\sigma(x)) = \min\{j \in \mathbb Z_+:\sum_{i= 0 }^{j} A(x,m,i) \ge \alpha \}.$ The  function $\text{CVaR}_{\alpha}(\sigma(x))$ is given by  
 \begin{align*}
 \text{CVaR}_{\alpha}(\sigma(x)) &=  \text{VaR}_{\alpha}(\sigma(x)) -\frac{1}{\alpha} \left(\sum_{j= 0 }^{\text{VaR}_{\alpha}(\sigma(x))-1}  \mathbb P( \sigma(x) = j) ( \text{VaR}_{\alpha}(\sigma(x)) -j)  \right) \\
 &=   \text{VaR}_{\alpha}(\sigma(x))  \left(1 - \frac{1}{\alpha}\sum_{j= 0 }^{\text{VaR}_{\alpha}(\sigma(x))-1} A(x,m,j)) \right)+ \frac{1}{\alpha}  \sum_{j= 0 }^{\text{VaR}_{\alpha}(\sigma(x))-1} jA(x,m,j).
 \end{align*} 
 For a given $x$, the running time of the DP is $\mathcal{O}(nm+m^2)$, because obtaining $P(x,j)$ for all $j \in V_2$ is  $\mathcal{O}(nm)$, and computing the recursion is $\mathcal{O}(m^2)$. 
\end{proof}

\begin{table}[htb]
	\caption{ Algorithm \ref{alg:DCG_2} with different inequalities}
	\label{Table:IC_Oracle_1}
	\begin{center}
		\scalebox{0.65}{\begin{tabular}{ 			 |p{1cm}p{1cm}p{1cm}||p{1.3cm}p{1.3cm}p{1.3cm}||p{1.3cm}p{1.3cm}p{1.3cm}||p{1.3cm}p{1.3cm}p{1.3cm}||p{1.3cm}p{1.3cm}p{1.3cm}|}
				\hline				
				& & & \multicolumn{3}{c||}{Oracle-LShape} &
				\multicolumn{3}{c||}{Oracle-Ineq \eqref{eq:new_cut}} &
				\multicolumn{3}{c||}{Oracle-Ineq \eqref{eq:relax_cut_uplift}} & \multicolumn{3}{c|}{Oracle-Ineq \eqref{eq:new_cut} and \eqref{eq:relax_cut_uplift}} \\				
				\cline{4-15}				
				{$|\mathcal V|$}  &  { $\alpha$} & { $k$}
				& {Time } & {Cuts } & {Nodes }  & {Time } & {Cuts } & {Nodes } & {Time } & {Cuts } & {Nodes }  & {Time } & {Cuts } & {Nodes } 
				\\			
				\hline			
				50 & 0.025  & 3 & 9 & 2315 & 2363 & 3 & 801 &1915& $\le 1$ & 224&336 & $\le 1$ & 398 &667 \\
				50 & 0.025  & 5 & $\ge 1800$ & 33988 & 36405& 681&16683&52677 &3&726&1960& 5&1462&2264\\
				50 & 0.05  & 3 & 10 & 2319 & 2363 & 2 & 799 &1734&$\le 1$ & 225 &366 &$\le 1$ &397&787\\
				50 & 0.05  & 5 & $\ge 1800$ & 32548 & 36569 & 492& 14978 & 43320& 3 &713&1859& 5& 1484& 2095
				\\
				100 & 0.025  & 3 & 1152 &19661 &20300 & 11& 1845 &4287& 24&1511&2321&15&1816&1252\\
				100 & 0.025  & 5 & $\ge 1800$ &21834& 53054&1154&17741&97510&825&13156&46789&560&15648&45091\\
				100 & 0.05  & 3 & 1205 & 19634 &20685 &9&1486&4476&24&1542&2400&15&1808&1224\\
				100 & 0.05  & 5 & $\ge 1800$ &21980&52479 &	505& 11408&74808&826&12911&47823&238&10684&21191
				\\
				150 & 0.025  & 3 &  $\ge 1800$ &18922&22045 &25&1832&12753& 457&6775&10868& 91&3046&1389\\
				150 & 0.025  & 5 &  $\ge 1800$ &19364&43796 &	1640&17781&109795&$\ge 1800$ & 13468&72593&1603&20278&92923\\
				150 & 0.05  & 3 & $\ge 1800$ &20394&21649 &6&1021&1115& 445&6705&10815&	59&1802&8489\\
				150 & 0.05  & 5 &	 $\ge 1800$ &20609&49523& $\ge 1800$ &10832&150285&$\ge 1800$& 14759&68158&1582&19978&75957\\					
				\hline			
		\end{tabular}}
	\end{center}
\end{table}	
 
We study the computational performance of our proposed methods on RASC. All instances were executed on a Windows 8.1 operating system with an Intel Core i5-4200U 1.60 GHz CPU, 8 GB DRAM, and x64 based processor using C++ with IBM ILOG CPLEX 12.7. We set up the mixed-integer programming search method as traditional branch-and-cut with the lazycallback function of CPLEX, where the presolve process is turned off. The number of threads is equal to one. All other CPLEX options are set to their default values.  
The time limit is set to 1800 seconds. For RASC, we generate a complete bipartite graph with arcs from all nodes $ i \in V_1$ to all $j \in V_2$. Let $\mathcal V = V_1 \cup V_2$. (Note that $V=V_1$ for RASC, because we only select nodes from $V_1$.) We follow the data generation scheme of \cite{Wu2017} to set $n$, $m$, and $a_{ij}$ for each arc $(i, j)$. We let  $k \in \{3,5\}$ and $\alpha \in \{0.025, 0.05\}$. The size of the  bipartite graphs is $|\mathcal V| \in \{50,100,150\}$.

In Table \ref{Table:IC_Oracle_1}, we  demonstrate the performance of Algorithm \ref{alg:DCG_2},  referred to  as ``Oracle", using three types of valid inequalities, the L-shaped cut of \cite{LL93}, inequality \eqref{eq:new_cut} and lifted inequality \eqref{eq:relax_cut_uplift}. Note that for a given incumbent solution, $\bar x$, which is a characteristic vector of the set $\bar X$, we consider the L-shaped cut
\begin{equation}\label{eq:Lcut}
\psi \le \text{CVaR}_{\alpha}(\sigma(\bar x)) +\sum_{j\in V\setminus \bar X} \big (\text{CVaR}_{1}(\sigma(V))-\text{CVaR}_{\alpha}(\sigma(\bar x))\big) x_j.
\end{equation} 
For $x_j=1$ for any $j\in V\setminus \bar X$, inequality \eqref{eq:Lcut} gives a valid upper bound for $\psi=\text{CVaR}_{\alpha}(\sigma(\bar x))\big)$, given by  $\text{CVaR}_{1}(\sigma(\bar x))\big)$  for any $\alpha \in (0,1]$,  from  Observation \ref{prop:CVaR_sub} (ii). Clearly, inequalities \eqref{eq:new_cut} are stronger than   inequalities \eqref{eq:Lcut}, because $\text{CVaR}_{1}(\sigma(\bar x+\mathbf {e}_{j}))\le\text{CVaR}_{1}(\sigma(V))$, for $j\in V\setminus \bar X$.  Column ``Oracle-LShape" denotes Algorithm \ref{alg:DCG_2} with the L-shaped cut \eqref{eq:Lcut}. Column ``Oracle-Ineq \eqref{eq:new_cut}" denotes Algorithm \ref{alg:DCG_2} with inequality \eqref{eq:new_cut}. Column ``Oracle-Ineq \eqref{eq:relax_cut_uplift}" denotes Algorithm \ref{alg:DCG_2} with inequality \eqref{eq:relax_cut_uplift}. Column ``Oracle-Ineq \eqref{eq:new_cut} and \eqref{eq:relax_cut_uplift}" denotes Algorithm \ref{alg:DCG_2} with both inequalities \eqref{eq:new_cut} and \eqref{eq:relax_cut_uplift}. Column ``Time" denotes the total solution  time for  each instance for RMP, in seconds. Column ``Cuts" denotes the total number of user cuts added to RMP. Column ``Nodes" denotes the number of branch-and-bound nodes traced in RMP. From \tabref{Table:IC_Oracle_1}, we observe  that the solution time increases as  $k$ and $|\mathcal V|$ increase for all methods. We observe that Oracle-LShape cannot solve most instances  within the time limit. Oracle-Ineq \eqref{eq:new_cut} and Oracle-Ineq \eqref{eq:relax_cut_uplift} are faster than Oracle-LShape for the instances that  are solvable by Oracle-LShape within the time limit. In addition, Oracle-Ineq \eqref{eq:new_cut} or Oracle-Ineq \eqref{eq:relax_cut_uplift} generates a fewer number of optimality cuts and traces a fewer number of nodes compared to Oracle-LShape.

For most of the  instances with $|\mathcal V| \le 100$, Oracle-Ineq \eqref{eq:relax_cut_uplift} generates a fewer number of optimality cuts and traces a fewer number of nodes compared to Oracle-Ineq \eqref{eq:new_cut}. However, for the instances with   $|\mathcal V| = 150$, the solution time of Oracle-Ineq \eqref{eq:relax_cut_uplift} is more than Oracle-Ineq \eqref{eq:new_cut}. Recall that for a given incumbent solution $\bar x$, inequality \eqref{eq:relax_cut_uplift} is generated by Algorithm \ref{alg:greedy_uplift}. In Algorithm \ref{alg:greedy_uplift}, we observe that for a given $\bar x$ and $1 \le i \le i' \le r$, we have 
$\bar \delta_{j_i}(\bar x) \le \bar \delta_{j_i'}(\bar x) \le \text{CVaR}_{\alpha}(\sigma(V_1))- \text{CVaR}_{\alpha}(\sigma(\bar x)) .$ From the above relation, if the size of $|\mathcal V|$ is large, then there may exist a large number of nodes with a high value of the lifting function $\delta_{j_i}(\bar x)$, which is close to the coefficients in the L-shaped cut \eqref{eq:Lcut}, i.e. $\text{CVaR}_{\alpha}(\sigma(V_1))- \text{CVaR}_{\alpha}(\sigma(\bar x))$. This explains why Oracle-Ineq \eqref{eq:relax_cut_uplift} does not perform well in instances with  large $|\mathcal V|$. To benefit from the complementary strengths of inequalities \eqref{eq:new_cut} and \eqref{eq:relax_cut_uplift}, we add both classes of inequalities at each iteration in Algorithm \ref{alg:greedy_uplift}. In \tabref{Table:IC_Oracle_1}, we observe that for the instances that  are not  solvable by either Oracle-Ineq \eqref{eq:relax_cut_uplift} or Oracle-Ineq \eqref{eq:new_cut} within 1800 seconds, the solution time of Oracle-Ineq \eqref{eq:new_cut} and \eqref{eq:relax_cut_uplift} is shorter. Furthermore, for a hard instance  $(|V|,\alpha,k)= (150,0.05,5)$, only Oracle-Ineq \eqref{eq:new_cut} and \eqref{eq:relax_cut_uplift}  provides an optimal solution within the time limit. In conclusion, our proposed inequalities enable the effective solution of difficult instances of the RASC problem that cannot be solved with existing methods within the set time limit.

\section*{Acknowledgments}
We thank the AE and the reviewer for their comments that improved the presentation. This work is supported, in part, by the National Science
Foundation Grant \#1907463.

\bibliographystyle{abbrvnat}
\setlength{\bibsep}{0.0pt}
	\bibliography{general}
	
\end{document}